\documentclass{article}

\usepackage{enumerate, longtable}
\usepackage{amsmath, amscd, amsfonts, amsthm, amssymb, latexsym, comment, stmaryrd, graphicx, dsfont, mathtools, authblk}
\usepackage{xcolor}
\usepackage[all]{xy}
\usepackage{fullpage}
\usepackage{tikz}
\usepackage[
        colorlinks,
        linkcolor=red,  citecolor=blue,
        backref
]{hyperref}

\usepackage{amstext} % for \text macro
\usepackage{array}
\usepackage{cleveref}

\newtheorem{claim}{Claim}[section]

\newtheorem{thm}{Theorem}
\newtheorem{prop}[claim]{Proposition}
\newtheorem{cor}[claim]{Corollary}
\newtheorem{lem}[claim]{Lemma}

\theoremstyle{definition}

\newtheorem{defi}[claim]{Definition} 
\newtheorem{rem}[claim]{Remark}

\newcommand{\F}{{\mathbb F}}
\newcommand{\Z}{{\mathbb Z}}

\newcommand{\R}{{\mathbb R}}

\renewcommand{\P}{{\mathbb P}}

\newcommand{\isom}{{\xrightarrow{\sim}}}

\newcommand{\con}{{\mathrm{con}}}
\newcommand{\const}{{\mathrm{const}}}
\newcommand{\mon}{{\mathrm{mon}}}
\newcommand{\Gal}{{\mathrm{Gal}}}
\newcommand{\disc}{{\mathrm{disc}}}
\newcommand{\res}{{\mathrm{res}}}
\newcommand{\lc}{{\mathrm{lc}}}
\renewcommand{\Cup}{\bigcup}
\newcommand{\bs}{\boldsymbol}

\newcommand{\ul}{\underline}
\newcommand{\mf}{\mathfrak}
\newcommand{\mc}{\mathcal}

\numberwithin{equation}{section}

\begin{document}

\title{Galois groups of random polynomials over the rational function field}
\author{Alexei Entin\thanks{aentin@tauex.tau.ac.il}}
\affil{School of Mathematical Sciences, Tel Aviv University, Tel Aviv 69978, Israel}
%\email{aentin@tauex.tau.ac.il}

\date{}
\maketitle

\begin{abstract} For a fixed prime power $q$ and natural number $d$ we consider a random polynomial $$f=x^n+a_{n-1}(t)x^{n-1}+\ldots+a_1(t)x+a_0(t)\in\F_q[t][x]$$ with $a_i$ drawn uniformly and independently at random from the set of all polynomials in $\F_q[t]$ of degree $\le d$. We show that with probability tending to 1 as $n\to\infty$ the Galois group $G_f$ of $f$ over $\F_q(t)$ is isomorphic to $S_{n-k}\times C$, where $C$ is cyclic, $k$ and $|C|$ are small quantities with a simple explicit dependence on $f$. As a corollary we deduce that $\P(G_f=S_n\,|\,f\mbox{ irreducible})\to 1$ as $n\to\infty$. Thus we are able to overcome the $S_n$ versus $A_n$ ambiguity in the most natural restricted coefficients random polynomial model over $\F_q[t]$, which has not been achieved over $\Z$ so far.
\end{abstract}

\section{Introduction}

Consider a random polynomial \begin{equation}\label{eq: small box basic}f=x^n+a_{n-1}x^{n-1}+\ldots+a_1x+a_0\in\Z[x],\quad a_i\in\{1,\ldots,H\},\end{equation} with each $a_i$ drawn uniformly and independently from $\{1,\ldots,H\}$. The question of how often such a polynomial is reducible or has a Galois group $G_f$ smaller than $S_n$ has been widely studied (along with variations where the coefficients are taken from $\{-H,\ldots,H\}$ or other distributions). The two main regimes considered in the literature are when we fix $n$ and take $H\to\infty$, called the \emph{large box model} (or the fixed degree model), and when we fix $H$ and take $n\to\infty$, called the \emph{restricted coefficients model}. 

In the large box model it was shown by van der Waerden \cite{VDW36} that $\lim_{H\to\infty}\P(G_f\neq S_n)=0$, and he conjectured that in fact $\P(G_f\neq S_n)=O(H^{-1})$ as $H\to\infty$. After a series of incremental results due to Knobloch \cite{Kno55}, Gallagher \cite{Gal73}, Dietmann \cite{Die12,Die12a}, Chow-Dietmann \cite{ChDi23} and Anderson et al \cite{AGLLS23} the conjecture was finally established by Bhargava \cite{Bha21_}, who was able to treat the remaining most difficult case of $G_f=A_n$ (the alternating group). Van der Waerden also made the more precise conjecture $\P(G_f\neq S_n,S_{n-1})=o(H^{-1})$, which remains open. Here once again the remaining case is to show that $\P(G_f=A_n)=o(H^{-1})$ (except for $n=7,8$ where additional cases remain). For $\P(G_f\neq S_n,A_n)$ an improved bound was obtained by Lemke-Oliver \cite[Corollary 1.3]{Lem23_2}. For recent results on the related problem of counting number fields of fixed degree and Galois group ordered by discriminant see \cite{ChDi23a_1, Lem23_2, Lem24_1}.

The restricted coefficients model received considerable attention over the last three decades, with important breakthroughs made in recent years. A folklore conjecture is that for fixed $H\ge 2$ we have $\lim_{n\to\infty}\P(G_f=S_n)=1$.
The weaker result $\lim_{n\to\infty}\P(G_f=S_n\mbox{ or }A_n)=1$ was first established 
by Bary-Soroker and Kozma \cite{BaKo20} for 
$H$ divisible by at least four 
distinct primes (e.g. 
$H=210$) and later by 
Bary-Soroker, Kozma and 
Koukoulopoulos 
\cite{BKK23} for all $H\ge 35$. Conditional on the Extended Riemann Hypothesis (for Dedekind zeta functions) Breuillard and Varj\'u \cite{BrVa19} established $\lim_{n\to\infty}\P(G_f=S_n\mbox{ or }A_n)=1$ for any $H\ge 2$. However showing that 
$A_n$ occurs with 
negligible probability 
in the restricted coefficients model 
proved to be a 
formidable challenge. A similar difficulty in eliminating $A_n$ occured in the work of Eberhard \cite{Ebe22} on the Galois group of characteristic polynomials of random matrices.

The reason that the group $A_n$ poses a unique challenge (compared to other transitive subgroups of $S_n$) is that it is highly transitive (the method of Breuillard-Varj\'u works by showing that $G_f$ is highly transitive with probability close to 1) and the cycle structure of its elements looks quite ``random" except for the parity condition (the methods of Bary-Soroker--Kozma--Koukoulopoulos use the fact that the elements of a transitive subgroup $A_n\neq G< S_n$ have very special cycle structures). The natural approach to estimating $\P(G_f=A_n)$ in the restricted coefficients model is to apply the discriminant criterion and estimate $\P(\disc(f)=\square)$. However the discriminant is a complicated expression in the coefficients of $f$ (especially in the large degree limit) and it is typically very hard to determine how often it is a square. 

In some variations of the problem this challenge can be overcome. In the case of Littlewood 
polynomials (random polynomials with 
coefficients $\pm 1$) 
of degree $n\equiv 2,4\pmod 8$ 
 one always has $\disc(f)\equiv 5\pmod 8$ (see \cite[\S 4]{BKK23}), so $\disc(f)$ can never be a square and $\lim_{n\to\infty}\P(G_f=A_n)=0$ in this model.
In the case of palindromic Littlewood polynomials 
\cite{Hok23_1} one has $\disc(f)=(-1)^nf(1)f(-1)\cdot\square$ and using this one can show $\lim_{n\to\infty}\P(G_f\subset A_n)=0$ \cite[Corollary 1.2]{Hok23_1} (the precise typical Galois group has not been established in this model, though this might be accessible to existing methods under the ERH). In these and similar variations it is  easy to tell when 
$\disc(f)$ is not a square, which is not the case for the more basic model defined by (\ref{eq: small box basic}).

The weaker statement $\lim_{n\to\infty}\P(f\mbox{ is irreducible})=1$ (first conjectured by Odlyzko and Poonen \cite{OdPo93} for a variation with $a_i\in\{0,1\}\,(1\le i\le n-1),\,a_0=1$) is also only known for $H\ge 35$ unconditionally and for all $H\ge 2$ under the ERH, by the same machinery. Recently this has been established unconditionally for the Littlewood model $a_i=\pm 1$ for a sparse sequence of values of $n$ \cite{BHKP24_2}. A variation with characteristic polynomials of random symmetric matrices was established in \cite{FJSS23} under ERH.

In \cite{EnPo23_2} the author and Popov studied the function field analog of the restricted coefficients model, which we now describe (a different function field model involving additive polynomials was studied by the author, Bary-Soroker and McKemmie \cite{BEM24}). Let $q$ be a fixed prime power and $d$ a fixed natural number. Denote
$$\F_q[t]_{\le d}=\{h\in\F_q[t]:\,\deg h\le d\}$$ and consider the random polynomial
\begin{equation}\label{eq: small box ff}f=x^n+a_{n-1}(t)x^{n-1}+\ldots+a_1(t)x+a_0(t)\in\F_q[t][x],\quad a_i\in\F_q[t]_{\le d},\end{equation} with
$a_i$ drawn uniformly and independently from $\F_q[t]_{\le d}$. For a polynomial $f\in\F_q(t)[x]$ we denote by $G_f$ the Galois group of its splitting field over $\F_q(t)$, provided that the irreducible factors of $f$ are separable (otherwise we leave $G_f$ undefined). We are interested in the distribution of $G_f$ where $f$ is a random polynomial as in (\ref{eq: small box ff}) and $n$ is large.

An important feature of the function field analog is that with positive probability $f$ has a nontrivial factor lying in $\F_q[x]$ and thus having a cyclic Galois group, preventing $G_f=S_n$. This phenomenon likely persists in any reasonable variation of the model (e.g. if we require each $a_i$ to lie in some subset $S_i\subset\F_q[t]_{\le d}$ with $|S_i|\ge 2$). Indeed we may rewrite $f=\sum_{j=0}^db_j(x)t^j\in\F_q[x][t]$ and denote $$c(x)=\con_t(f)=\gcd(b_0,\ldots,b_d)\in\F_q[t]$$ (the content of $f$ in the variable $t$), so that $f(t,x)=c(x)g(t,x)$ where $\con_t(g)=1$.
It is reasonable to expect that the GCD
of a fixed number of
random polynomials
in $\F_q[x]$ is 
$\neq 1$ with 
positive 
probability, and 
indeed it is shown 
in \cite[Proposition 1.9]{EnPo23_2} that 
for a random polynomial $f$ as in (\ref{eq: small box ff})
and any monic 
$c\in\F_q[x],\deg c=k$
we have 
\begin{equation}\label{eq: prob content}\P\left(\con_t(f)=c\right)=\left(1-\frac 1{q^d}\right)\frac 1{q^{k(d+1)}}\end{equation} whenever $n>k$.

If $\con_t(f)=c$ and $g(t,x)\in\F_q(t)[x]$ is separable (the latter condition holds with limit probability 1 by \cite[Proposition 7.3]{EnPo23_2}) then the splitting field of $c(x)$ over $\F_q(t)$, which has the form $\F_{q^m}(t)$ and thus has a cyclic Galois group $C$, while the Galois group of the splitting field of $g$ can be viewed as a subgroup of $S_{n-k}$ via its action on the roots of $g$. Hence $G_f$ can be viewed as a subgroup of $S_{n-k}\times C$. 

With this notation, it was shown in \cite[Theorem 1.7]{EnPo23_2} that 
\begin{equation}\label{eq: thm popov}\lim_{n\to\infty}\P\left(G_f\in\{S_{n-k}\times C,A_{n-k}\times C,S_{n-k}\times_{C_2}C\}\right)=1\end{equation}
(the third possibility $S_{n-k}\times_{C_2}C$ occurs only if $|C|$ is even and denotes the fiber product with respect to the unique nontrivial homomorphisms to the group $C_2$ of order 2) and that assuming a uniform variant of the polynomial Chowla conjecture over $\F_q[t]$ one has $\lim_{n\to\infty}\P(G_f=S_{n-k}\times C)=1$. In the present paper we show this unconditionally, thus resolving the ambiguity in (\ref{eq: thm popov}), which is analogous to the $S_n$ versus $A_n$ ambiguity in the restricted coefficients model over $\Z$.

\begin{thm}\label{thm: main} Let $q$ be a fixed odd prime power, $d$ a fixed natural number. Let $f=x^n+\sum_{i=0}^{n-1}a_i(t)x^i$ be a random polynomial with each $a_i\in\F_q[x]_{\le d}$ drawn uniformly and independently. Then
$$\lim_{n\to\infty}\P(G_f= S_{n-k}\times C)=1,$$
where $k=\deg_x\con_t(f)$ and $C=\Gal(\con_t(f)/\F_q(t))=\Gal(\con_t(f)/\F_q)$ is a cyclic group of order\\ $\mathrm{lcm}(\deg P_1,\ldots,\deg P_l)$, with $\con_t(f)=\prod_{i=1}^lP_i$ being the prime factorization of $\con_t(f)$ in $\F_q[x]$.
\end{thm}

\begin{cor}\label{cor: main} Let $q$ be a fixed odd prime power, $d$ a fixed natural number. Let $f=x^n+\sum_{i=0}^{n-1}a_i(t)x^i$ be a random polynomial with each $a_i\in\F_q[x]_{\le d}$ drawn uniformly and independently. Then $$\lim_{n\to\infty}\P(G_f=S_n\,|\,f\mbox{ is irreducible})=\lim_{n\to\infty}\P(G_f=S_n\,|\,\con_t(f)=1)=1.$$
\end{cor}

\begin{proof} Immediate from Theorem \ref{thm: main} and the fact that $\lim_{n\to\infty}\P(f\mbox{ is irreducible})=\lim_{n\to\infty}\P(\con_t(f)=1)=1-\frac 1{q^d}>0$, which follows from (\ref{eq: thm popov}) and (\ref{eq: prob content}).
\end{proof}

\begin{rem} One may argue that a more natural analog of the model (\ref{eq: small box basic}) over $\F_q[t]$ would be to take the coefficients $a_i(t)$ to be monic of degree $d$ or monic of degree up to $d$. Our methods, as well as the methods of \cite{EnPo23_2}, can be adapted with minimal changes to prove the assertion of Theorem \ref{thm: main} for the variation with $a_i$ monic of degree $d$, chosen uniformly and independently. The distribution of the values of $k$ and $|C|$ is different in this model, but it can also be easily computed. We chose the model (\ref{eq: small box ff}) for convenience and because it is the one studied in \cite{EnPo23_2}. The model with $a_i$ monic of degree up to $d$, as well as other non-uniform distributions on $\F_q[t]_{\le d}$, would require a substantial enhancement of our methods.
\end{rem}

\begin{rem} It is possible to obtain a more explicit estimate in Theorem \ref{thm: main} of the form $P(G_f=S_{n-k}\times C)=1+O(n^{-\delta})$ for some $\delta=\delta(q,d)>0$. We chose not to give the details here, as it would require reworking the results of \cite{EnPo23_2} (as well as \cite[Lemma 9]{BaKo20}, a key result used therein) with more explicit error terms, which would significantly lengthen the exposition. Theorem \ref{thm: disc} below, which is the main new ingredient in the present work, will be stated with an explicit error term of this form.
\end{rem}

\begin{rem} Our method crucially depends on the assumption that $q$ is odd, as it uses the discriminant criterion for when the Galois group of a polynomial is contained in the alternating group, which does not apply in characteristic 2. The alternative criteria available in characteristic 2 lack the multiplicative structure we exploit in the present work, so this case would require entirely new methods.\end{rem}

\noindent
{\bf Acknowledgment.} The author would like to thank the anonymous referee of a previous draft of this paper for suggesting references to related work and for spotting a minor error. The author was partially supported by Israel Science Foundation grant no. 2507/19.

\section{Outline of the proof}

Theorem \ref{thm: main} will follow by combining the results of \cite{EnPo23_2} and the following

\begin{thm}\label{thm: disc}Let $q$ be a fixed odd prime power, $d$ a fixed natural number. Let $f=x^n+\sum_{i=0}^{n-1}a_i(t)x^i$ be a random polynomial with each $a_i\in\F_q[t]_{\le d}$ drawn uniformly and independently. Then
$$\P(\disc_x(f)=\const\cdot\square)=O_{q,d}\left(\frac{\log^{2}n}{n^{1/3}}\right),$$
where $\const\cdot\square$ denotes an element of $\F_q[t]$ which is of the form $au^2$ with $a\in\F_q,\,u\in\F_q[t]$. The implicit constant may depend on $q,d$.
\end{thm}

To explain the basic idea of the derivation of Theorem \ref{thm: main} from Theorem \ref{thm: disc}, consider for simplicity a random polynomial $f$ as in (\ref{eq: small box ff}) conditional on $c(x)=\con_t(f)=1$, which by (\ref{eq: prob content}) occurs with probability $1-\frac 1{q^d}$. Then $C=G_c=1$ and by (\ref{eq: thm popov}) with limit probability 1 we have $G_f\in\{S_n,A_n\}$. However $G_f=A_n$ cannot occur if $\disc(f)\neq\square$, which by Theorem \ref{thm: disc} happens with limit probability 1. The general argument is contained in \cite[Proof of Theorem 1.8]{EnPo23_2}.

\begin{proof}[Proof of Theorem \ref{thm: main}] In \cite[Proof of Theorem 1.8]{EnPo23_2} the assertion of Theorem \ref{thm: main} is derived from the assertion of Theorem \ref{thm: disc}. In that work the latter assertion is only proved conditionally on a uniform polynomial Chowla conjecture \cite[Conjecture 1.5]{EnPo23_2}, but here we will prove it unconditionally, thus establishing Theorem \ref{thm: main} unconditionally.\end{proof}

The rest of the paper is dedicated to proving Theorem \ref{thm: disc}. 

\subsection{Notation and conventions}

Throughout the rest of the paper $p>2$ is a fixed prime, $q$ is a fixed power of $p$ and $d$ is a fixed natural number. All asymptotic notation will have implicit constants or rate of convergence which may depend on $q,d$, while dependence on other parameters will be stated explicitly. We will occasionally reiterate these assumptions for emphasis.

For a polynomial $f\in\F_q[t,x]$ we will always denote by $\lc(f)\in\F_q[t]$ its leading coefficient as a polynomial in $x$, by $f'=\frac{\partial f}{\partial x}\in\F_q[t,x]$ its derivative in the variable $x$ and  by $\disc(f)\in\F_q[t]$ its discriminant in the variable $x$. For any $e\in\R$ we denote
$$\F_q[x]_{\le e}=\{h\in\F_q[x]:\deg h\le e\},$$
%$$\F_q[x]_{e}=\{h\in\F_q[x]:\deg h= e\},$$
$$\F_q[x]^\mon_{e}=\{h\in\F_q[x]\mbox{ monic}:\deg h= e\}.$$
The cardinality of a finite set $S$ is denoted by $|S|$. If $P\in\F_q[t]$ is a prime and $Q\in\F_q[t]$ then $P^k\Vert Q$ means that $P$ divides $Q$ exactly $k$ times. The algebraic closure of a field $K$ is denoted by $\overline K$. By an absolutely irreducible polynomial $f\in\F_q[t,x]$ we mean a polynomial which is irreducible over $\overline\F_q$. If $A$ is a ring and $a\in A$ we denote by $A/a=A/(a)$ the quotient of $A$ by the ideal generated by $a$.

\subsection{Outline of the proof of Theorem \ref{thm: disc}}\label{sec: outline thm2}

Our basic strategy is to subdivide the family of polynomials (\ref{eq: small box ff}) into subfamilies with a fixed value of $f'$, along which $\disc(f)$ behaves more regularly. A similar strategy underlies the work of Sawin and Shusterman on the Hardy-Littlewood, Chowla and Bateman-Horn conjectures over function fields \cite{SaSh22a,Sash22} as well as the work on squarefree values of polynomials over function fields by Poonen \cite{Poo03}, Lando \cite{Lan15} and Carmon \cite{Car21}. 

First we recall (see e.g. \cite[Equation (2.4)]{BEF23}) that for a monic $f\in\F_q[t][x]$ with $f'\neq 0$ we have
\begin{equation}\label{eq: disc prod basic}\disc(f)=\pm\lc(f')^{\deg f}\prod_{f'(\alpha)=0}f(\alpha),\end{equation}
where the product is over the roots of $f'$ in $\overline{\F_q(t)}$ counted with multiplicity (if $f'=0$ then $\disc(f)=0$).

The main idea of our proof is to subdivide the family of polynomials
\begin{equation}\label{eq: def F}\mathcal F=\Biggl\{x^n+\sum_{i=0}^{n-1}a_ix^i:\,a_i\in\F_q[t]_{\le d}\Biggr\}\end{equation} into subfamilies of the form
\begin{equation}\label{eq: def F_f0}\mathcal F_{f_0}=\bigl\{f_0(t,x)+h(x)^p:\,h\in\F_q[x]_{\le(n-1)/p}\bigr\}\end{equation}
with $f_0\in\mathcal F$. Note that for a given $f_0$ the derivative of any $f\in\mathcal F_{f_0}$ is $f'=f_0'+(h(x)^p)'=f_0'$ and therefore if $f_0'\neq 0$ we have (from (\ref{eq: disc prod basic}))
\begin{equation}\label{eq: disc prod over roots}\disc(f)=\pm\lc(f_0')^n\prod_{f_0'(\alpha)=0}(f_0(\alpha)+h(\alpha)^p),\end{equation}
with the product being over the roots of $f_0'$.

Let us now assume for simplicity that $f_0$ is such that $f_0'(t,x)$ is absolutely irreducible (we will see later that with probability close to 1 as $f_0$ ranges uniformly over $\mathcal F$, $f_0'$ is close to being absolutely irreducible. We will also see how to overcome the discrepancy when present). Then by (\ref{eq: disc prod over roots}) we have \begin{equation}\label{eq: disc norm}\disc(f)=\pm\lc(f_0')^nN_{\F_q(t,\alpha)/\F_q(t)}(f_0(\alpha)+h(\alpha)^p),\end{equation}
where $\alpha$ is a root of $f$ in $\overline{\F_q(t)}$.
The set of elements $f_0(\alpha)+h(\alpha)^p,\,h\in\F_q[x]_{\le(n-1)/p}$ is sufficiently regular for us to extract useful information about the prime (divisor) factorization of its typical element, using an elementary but delicate sieve argument. Using (\ref{eq: disc norm}) we will be able to deduce that with probability close to 1, the discriminant $\disc(f)$ has a prime divisor $P$ of multiplicity 1 and therefore $\disc(f)\neq\const\cdot\square$.

For technical reasons we impose two additional conditions on $f_0$, which we will show are satisfied with probability close to 1. Namely, we require that $\Gal(f_0'(t,x)/\overline\F_q(x))=S_d$ and that $f_0$ satisfies the following
\\ \\
{\bf Condition ($\bigstar$).} Write $f_0'=\sum_{i=0}^db_i(x)t^i$. Then either $d=1$, or $b_1\neq 0$ and
$\frac{b_1}{(b_1,b_2)}\not\in\F_q[x^p]$.
\\ \\
We now state precisely the two main steps outlined above.

\begin{prop}\label{prop: derivative} Assume $q$ is odd. Let $f_0\in \mathcal F$ be uniformly random. If $f_0'\neq 0$ write $f_0'=c(x)g(t,x)$ with $c=\con_t(f_0')$. Then
\begin{multline}\P\left(f_0'\neq 0,\,g\mbox{ is absolutely irreducible and separable in }t,\,\Gal(g/\F_q(x))=S_d,\,(\bigstar)\mbox{ holds},\,\deg c<n^{1/2}\right)\\=1+O\left(q^{-n^{1/3}}\right).\end{multline}
\end{prop}

\begin{prop}\label{prop: disc} Assume $q$ is odd. Let $f_0\in\mathcal F$ be such that $f_0'=c(x)g(t,x),\,c=\con_t(f_0),\,\deg c< n^{1/2}$, $g$ is absolutely irreducible and separable in $t$ with $\Gal(g/\F_q(t))=S_d$ and ($\bigstar$) holds. Let $h\in\F_q[x]_{\le(n-1)/p}$ be uniformly random and denote $f(t,x)=f_0(t,x)+h(x)^p$. Then 
$$\P\left(\exists P\in\F_q[t]\mbox{ prime such that }n^{2/3}\le\deg P\le \frac{n}{3p\log^2n}, P\Vert\disc(f)\right)=1+O\left(\frac{\log^2n}{n^{1/3}}\right).$$
Consequently $$\P(\disc(f)\neq\const\cdot\square)=1+O\left(\frac{\log^2n}{n^{1/3}}\right).$$
\end{prop}

We now show how propositions \ref{prop: derivative} and \ref{prop: disc} imply Theorem \ref{thm: disc}.

\begin{proof}[Proof of Theorem \ref{thm: disc}]

Let $\mathcal E$ be the set of $f_0\in\mathcal F$ satisfying the condition of Proposition \ref{prop: disc}. 
By Proposition \ref{prop: derivative}  we have \begin{equation}\label{eq: proof thm2 1}|\mathcal F\setminus\mathcal E|=O\left(q^{-n^{1/3}}|\mathcal F|\right).\end{equation}
For any $f_0\in\mathcal E$, by Proposition \ref{prop: disc} we have (using the notation (\ref{eq: def F_f0}))
\begin{equation}\label{eq: proof thm2 2}\left|\{f\in\mathcal F_{f_0}:\,\disc(f)=\const\cdot\square\}\right|=O\left(\frac{\log^2n}{n^{1/3}}|\mathcal F_{f_0}|\right).\end{equation}
Each $f\in\mathcal F$ can be written in exactly $q^{\lfloor\frac{n-1}p\rfloor+1}$ ways as $f(t,x)=f_0(t,x)+h(x)^p$ where $f_0\in\mathcal F,\,h\in\F_q[x]_{\le(n-1)/p}$ (note that as $h$ ranges over $\F_q[x]_{\le(n-1)/p}$, the set of values of $h(x)^p$ coincides with the set of values of $h(x^p)$) and $|\mathcal F_{f_0}|=q^{\lfloor\frac{n-1}p\rfloor+1}$ for any $f_0\in\mathcal F$. Therefore (using (\ref{eq: proof thm2 1}) and (\ref{eq: proof thm2 2}))
\begin{multline*}q^{\lfloor\frac{n-1}p\rfloor+1}\left|\{f\in\mathcal F:\disc(f)=\const\cdot\square\}\right|=\sum_{f_0\in\mathcal F}\left|\{f\in\mathcal F_{f_0}:\disc(f_0)=\const\cdot\square\}\right|\\=\sum_{f_0\in\mathcal E}\left|\{f\in\mathcal F_{f_0}:\disc(f_0)=\const\cdot\square\}\right|+\sum_{f_0\in\mathcal F\setminus\mathcal E}\left|\{f\in\mathcal F_{f_0}:\disc(f_0)=\const\cdot\square\}\right|
\\ \le O\left(|\mathcal F| q^{\lfloor\frac{n-1}p\rfloor+1}\frac{\log^{2}n}{n^{1/3}}\right)+|\mathcal F\setminus\mathcal E| q^{\lfloor\frac{n-1}p\rfloor+1}\le O\left(|\mathcal F|q^{\lfloor\frac{n-1}p\rfloor+1}\frac{\log^{2}n}{n^{1/3}}\right)+O\left(q^{-n^{1/3}}|\mathcal F| q^{\lfloor\frac{n-1}p\rfloor+1}\right)\\
=O\left(|\mathcal F|q^{\lfloor\frac{n-1}p\rfloor+1}\left(\frac{\log^{2}n}{n^{1/3}}+q^{-n^{1/3}}\right)\right)=|\mathcal F|q^{\lfloor\frac{n-1}p\rfloor+1}\cdot O\left(\frac{\log^{2}n}{n^{1/3}}\right).
\end{multline*}
Dividing by $|\mathcal F|q^{\lfloor\frac{n-1}p\rfloor+1}$ gives $\P(\disc(f)=\const\cdot\square)=O\left(\frac{\log^{2}n}{n^{1/3}}\right)$, which is the assertion of the theorem.
\end{proof}

\begin{rem} Our arguments can be refined to show that for any sequence $\omega(n)\to\infty$, $\disc(f)$ has $\ge\frac{\log n}{\omega(n)}$ prime factors of multiplicity one with limit probability 1. Heuristically the precise number of such factors should be $\sim\log n$ with limit probability 1. We omit the details because a single such prime factor is sufficient for our purposes and we obtain a better error term for the weaker statement.
\end{rem}

\begin{rem} Proposition \ref{prop: derivative} essentially says that for a random $f\in\mathcal F$, the Galois group of $f'$ over $\F_q(x)$ is $S_d$ asymptotically almost surely. Here typically $\deg_tf'=d$ and $d$ is fixed, so this is a large box model with coefficients restricted to lie in the set of derivatives of polynomials in $\F_q[x]$. Thus the first step in our argument is a large box result. A similar situation occurs in the proof of (\ref{eq: thm popov}) in \cite{EnPo23_2}.\end{rem}

It remains to prove propositions \ref{prop: derivative} and \ref{prop: disc}. Proposition \ref{prop: derivative} will be proved in \S\ref{sec: derivative}. The proof of Proposition \ref{prop: disc}, which is the technical core of the present work, occupies \S\ref{sec: disc}-\ref{sec: conclusion}.

\section{Irreducibility of the primitive part of $f'$}
\label{sec: derivative}

In the present section we prove Proposition \ref{prop: derivative}. The main tool is an explicit Hilbert irreducibility theorem for $\F_q(t)$, which was proved in \cite{BaEn21}. Before recalling its statement we introduce some notation and recall some basic facts about the Galois group of a specialization of a polynomial. 
Let $p$ be a prime (in the present section we do not have to assume $p>2$ except in Lemma \ref{lem: lattice} and the proof of Proposition \ref{prop: derivative}), $q$ a power of $p$, $m,d$ natural numbers, $\bs y=(y_1,\ldots,y_m)$ a set of $m$ independent variables and $x,t$ additional independent variables. Throughout the present section by a separable polynomial we always mean separable in the variable $t$. Let $F\in\F_q[x,\bs y,t]$ be separable with $\deg_tF=d$ and denote $G=\Gal(F/\F_q(x,\bs y))$. Let $\bs\eta=(\eta_1,\ldots,\eta_m)\in\F_q[x]^m$ be such that $F(x,\bs\eta,t)$ is separable with $\deg_tF(x,\bs\eta,t)=d$ and denote $G_{\bs\eta}=\Gal(F(x,\bs\eta,t)/\F_q(x))$. Then one may view $G_{\boldsymbol\eta}$ as a subgroup of $G$ defined up to conjugation (this well-known fact can be deduced as a special case of \cite[Theorem VII.2.9]{Lan02}). Denote
$$\mathcal H_F=\{\bs\eta\in\F_q[x]^m:\,F(x,\bs\eta,t)\mbox{ separable of degree }d,\,G_{\bs\eta}=G\}.$$

\begin{prop}\label{prop: hilbert} 
Let $C>0$ be a constant, $F\in\F_q[x,\bs y,t]$ be separable with Galois group $G=\Gal(F/\F_q(x,\bs y))$ and $\deg_xF\le n^C$. With notation as above we have
$$\left|\mathcal H_F\cap\left(\F_q[x]_n^\mon\right)^m\right|=q^{nm}(1+O_{q,m,C,\deg_{\bs y,t}F}(nq^{-n/2})).$$
Here the implicit constant may depend on $q,m,C,\deg_{\bs y,t}F$ (the total degree of $F$ in the variables $y_1,\ldots,y_m,t$).

\end{prop}

\begin{proof} Follows immediately from \cite[Corollary 3.5]{BaEn21}.\end{proof}

We need a slight refinement of the above proposition.

\begin{prop}\label{prop: hilbert 2}
In the setup of Proposition \ref{prop: hilbert} let $\mathcal A_i\subset\F_q[x]_n^\mon,\,1\le i\le m$ be subsets with $|\mathcal A_i|\ge nq^{n/2}$. Then
$$|\mathcal H_F\cap (\mathcal A_1\times\cdots\times\mathcal A_m)|=\left(\prod_{i=1}^m|\mathcal A_i|\right)\cdot\left(1+O_{q,m,C,\deg_{\bs y,t} F}\left(\max_{1\le i\le m}\frac{nq^{n/2}}{|\mathcal A_i|}\right)\right).$$
The same is true if we require $\mathcal A_i\subset\F_q[x]_{\le n-1}$ instead.

\end{prop}

\begin{proof} The last assertion follows by applying the main assertion to $F(x,y_1+t^n,\ldots,y_m+t^n,t)$ instead of $F$. Hence we only need to deal with the case $\mathcal A_i\subset\F_q[x]_n^\mon$.

For the present proof we assume that $q,m,C,\deg_{\bs y,t}F$ are fixed and hence will be omitted from asymptotic notation. We will prove the assertion by induction on $m$. If $m=1$ we have (using Proposition \ref{prop: hilbert})
$$|\mathcal A_1|\ge|\mathcal H_F\cap \mathcal A_1|\ge|\mathcal A_1|-\left|(\F_q[x]^\mon_n\setminus\mathcal H_f)\right|=|\mathcal A_1|-O(nq^{n/2})=|\mathcal A_1|\left(1+O\left(\frac {nq^{n/2}}{|\mathcal A_1|}\right)\right)$$
and the assertion follows.

Assume $m\ge 2$ and that the assertion holds for $m-1$. We denote $$\hat{\bs y}=(y_2,\ldots,y_m),\quad\hat{\bs \eta}=(\eta_2,\ldots,\eta_m),\quad\widehat{\mathcal A}=\mathcal A_2\times\cdots\times \mathcal A_m.$$ Denote also $d=\deg_tF$.
First we claim that 
\begin{equation}\label{eq: spec}\bigl|\{\eta_1\in\F_q[x]_n^\mon:\,F(x,\eta_1,\hat{\bs y},t)\mbox{ not separable of degree $d$ or }\Gal(F(x,\eta_1,\hat{\bs y},T)/\F_q(x,\hat{\bs y}))\neq G\}\bigr|=O(nq^{n/2}).\end{equation}
Indeed for any $\eta_1$ from the set $E$ appearing on the LHS of (\ref{eq: spec}) and any $\hat {\bs \eta}\in(\F_q[x]^\mon_n)^{m-1}$ and $\bs\eta=(\eta_1,\hat{\bs \eta})$, either $F(x,\bs\eta,t)$ is not separable of degree $d$ or the group $G_{\bs \eta}\leqslant \Gal(F(x,\eta_1,\hat{\bs y},t)/\F_q(x,\hat{\bs y}))<G$ is smaller than $G$ and therefore $\bs\eta\not\in \mathcal H_F$. 
Consequently $(E\times(\F_q[x]^\mon_n)^{m-1})\cap \mathcal H_F=\emptyset$ and therefore by Proposition \ref{prop: hilbert} we must have $E=O(nq^{n/2})$ as claimed.

Now for any $\eta_1\not\in E$ we may apply the induction hypothesis to the polynomial $F(x,\eta_1,\hat{\bs y},t)$, which is separable of degree $d$ and has Galois group $G$, and conclude that
$$\Bigl|\mathcal H_F\cap \left(\{\eta_1\}\times\widehat{\mathcal A}\right)\Bigr|=|\widehat{\mathcal A}|\left(1+O\left(\max_{2\le i\le m}\frac{nq^{n/2}}{|\mathcal A_i|}\right)\right).$$
Summing over $\eta_1\in\mathcal A_1\setminus E$, noting that $|\mathcal A_1\setminus E|=1+O\left(\frac{nq^{n/2}}{|\mathcal A_1|}\right)$ (by (\ref{eq: spec})) and using the assumption $|\mathcal A_1|\ge nq^{n/2}$ we obtain
\begin{equation*}\begin{split}\prod_{i=1}^m|\mathcal A_i|\ge
|\mathcal H_F\cap (\mathcal A_1\times\cdots\times\mathcal A_m)|
&\ge|\widehat{\mathcal A}|\left(1+O\left(\max_{2\le i\le m}\frac{nq^{n/2}}{|\mathcal A_i|}\right)\right)\left(1+O\left(\frac{nq^{n/2}}{|\mathcal A_1|}\right)\right)\\&=\left(\prod_{i=1}^m|\mathcal A_i|\right)\cdot\left(1+O\left(\max_{1\le i\le m}\frac{nq^{n/2}}{|\mathcal A_i|}\right)\right).\end{split}\end{equation*}
The assertion follows.
\end{proof}

\begin{lem}\label{lem: disjoint}  Let $Q\in\F_q[t]$ be quadratic and irreducible, $H\in\F_q[x,t]$ separable with $\deg_tH=d$. Assume that $\Gal(H(x,t)Q(t)/\F_q(x))\cong S_d\times C_2$ ($C_2$ is the cyclic group of order 2). Then $H(x,t)=c(x)g(x,t)$ where $c=\con_t(H)$ and $g$ is absolutely irreducible and separable with $\Gal(g/\F_q(x))=S_d$.
\end{lem}

\begin{proof} Let $K,\F_{q^2}(x)$ be the respective splitting fields of $H,Q$ over $\F_q(x)$. Since $\Gal(Q/\F_q(x))=C_2,$ $\Gal(H/\F_q(x))\leqslant S_d$, our assumption implies that $\Gal(K/\F_q(x))=S_d$ and $K,\F_{q^2}(x)$ are linearly disjoint over $\F_q(x)$. Since the only cyclic quotient of $S_d$ is $C_2$, the algebraic closure of $\F_q$ in $K$ is contained in $\F_{q^2}$ and by the above linear disjointness it must equal $\F_q$. Consequently $K$ is linearly disjoint from $\overline \F_q$ over $\F_q$ and therefore $\Gal(H/\overline\F_q(x))=\Gal(H/\F_q(x))=S_d$ and $H$ is irreducible over $\overline \F_q(x)$. By Gauss's lemma we have $H(x,t)=c(x)g(t,x)$ with $c=\con_t(H)$ and $g$ irreducible in $\overline\F_q[x,t]$ with $\Gal(g/\F_q(x))=\Gal(H/\F_q(x))=S_d$, as required.
\end{proof}

\begin{lem}\label{lem: star} Let $f_0\in\mathcal F$ be uniformly random, where $\mathcal F$ is given by (\ref{eq: def F}). Then $\P((\bigstar)\mbox{ holds})=1+O(q^{-n/2})$ (Condition ($\bigstar$) was defined in \S\ref{sec: outline thm2}).
\end{lem}

\begin{proof} If $d=1$ Condition ($\bigstar$) is vacuous, so we assume $d>1$.
Write $f_0=\sum_{j=0}^db_j(x)t^j$. Then $b_1,b_{2}$ are independent and uniform in $\F_q[x]_{\le n-1}$. If ($\bigstar$) fails then we may write $b_1=bu,\,b_2=bv$ for some $u\in\F_q[x^p],v\in\F_q[x]$ and $b\in\F_q[x]$ monic. For a given $b$ with $\deg b=k$, there are at most $q^{\lfloor\frac{n-k-1}p\rfloor+1}$ values of $u$ and at most $q^{n-k}$ possible values $v$ such that $bu,bv\in\F_q[x]_{\le n-1}$. Given $\deg b=k$ there are $q^k-1$ possible values of $b$. Hence (using  $\P(b_{2}=0)=q^{-n}$)
\begin{multline*} 
\P((\bigstar)\mbox{ fails})=\frac 1{q^n}+
 \frac 1{q^{2n}}
 \left|\left\{(b_1,b_{2})\in\F_q[x]_{\le n-1}\times(\F_q[x]_{\le n-1}\setminus\{0\}):\frac{b_1}{\gcd(b_1,b_{2})}\in\F_q[x^p]\right\}\right|\\
 \le
 q^{-n}+q^{-2n}\sum_{k=0}^nq^kq^{\lfloor\frac{n-k-1}p\rfloor+1}q^{n-k}\ll q^{-n}\sum_{k=0}^nq^{\frac{n-k}p}\ll q^{-n}q^{n/p}\ll q^{-n/2}.
\end{multline*}
The assertion follows.

\end{proof}

\begin{lem}\label{lem: lattice} Assume $q$ is odd. Let $f_0\in\mathcal F$ be uniformly random, where $\mathcal F$ is given by (\ref{eq: def F}). Then
$$\P(\deg\con_t(f_0')\ge k)=O\left(q^{-\min(k,n/2)}\right).$$
\end{lem}

\begin{proof} The main tool for the proof of the lemma is \cite[Lemma 6.2]{EnPi23}, which gives the estimate 
\begin{equation}\label{eq: lattice}\left|\{h\in\F_q[x]_{\le n}:\,Q|h'\}\right|\ll
\left[\begin{array}{ll} q^{n-\deg Q}& 0\le \deg Q\le \frac n2,\\ q^{n-\deg Q +\frac 1p(2\deg Q-n)},& \frac n2<\deg Q<n.\end{array}\right. \end{equation}
for any $Q\in\F_q[x]_{\le n-1}$, where $p=\mathrm{char}(\F_q)>2$. This estimate is derived using the theory of lattices over $\F_q[x]$.

For $f_0\in\mathcal F$ write $f_0=\sum_{j=0}^db_j(x)t^j,\,b_j\in\F_q[x]_{\le n}$. Then $f_0'=\sum_{j=0}^db_j't^j$ and $\con_t(f_0')|b_j'$ for $0\le j\le d$. Therefore for $k<n$ (which we assume WLOG because $\deg\con_t(f_0')<n$) we have (using (\ref{eq: lattice})) 
\begin{multline*}\left|\{f_0\in\mathcal F:\,\deg\con_t(f_0)\ge k\}\right|\le \sum_{l=k}^{n-1}\sum_{c\in\F_q[x]_l^\mon}\left|\{b\in\F_q[x]_{\le n}:\,c|b'\}\right|^{d+1} 
\\ \le \sum_{l=k}^{\lfloor n/2\rfloor}q^lq^{(n-l)(d+1)}+\sum_{l=\lfloor n/2\rfloor+1}^{n-1}q^lq^{\left(n-l+\frac{2l-n}p\right)(d+1)}
\ll q^{n(d+1)-k}+q^{n(d+1)-n/2}\ll q^{n(d+1)}q^{-\min(k,n/2)}.
\end{multline*}
Dividing by $|\mathcal F|=q^{n(d+1)}$ we obtain the assertion of the lemma.

\end{proof}

\begin{proof}[Proof of Proposition \ref{prop: derivative}]
Set $p=\mathrm{char}(\F_q)>2$ and denote $$\mathcal D=\left\{h':\,h\in\F_q[x]\right\}=\left\{\sum_{i\ge 0} e_ix^i:a_i\in\F_q,\,a_i=0\mbox{ if }i\equiv -1\pmod p\right\}.$$
Let $f_0\in\mathcal F$ be a polynomial 
(with $\mathcal F$ defined by (\ref{eq: def F})). Then $f_0$ can be written uniquely as 
\begin{equation}\label{eq: f in t}f_0=\sum_{j=0}^db_j(x)t^j,\quad b_0\in\F_q[x]_n^\mon,\,b_j\in\F_q[x]_{\le n-1}\,(1\le j\le d).\end{equation}  
We have $$f_0'=\sum_{j=1}^db_j'(x)t^j,\quad b_j'\in\mathcal A_i,$$ where 
$$\mathcal A_0=\left\{nx^{n-1}+h:\,h\in\F_q[x]_{\le n-2}\right\}\cap\mathcal D,\,\quad\mathcal A_i=\F_q[x]_{\le n-2}\cap\mathcal D.$$ It is easy to see that the map
$$f\mapsto (b_0',\ldots,b_d'):\,\mathcal F\to\mathcal A_0\times\cdots\times\mathcal A_d$$ is $q^{(d+1)(\lfloor \frac{n-1}p\rfloor+1)}$-to-1. 

Let $Q\in\F_q[t]$ be any irreducible quadratic polynomial. We will now apply Proposition \ref{prop: hilbert 2} to the polynomial $$F=(y_dt^d+y_{d-1}t^{d-1}+\ldots+y_0)\cdot Q(t),$$ which has Galois group $G=\Gal(F/\F_q(x,y_0,\ldots,y_d))=S_d\times C_2$. If $\bs\eta\in\mathcal H_F$ then $F(\bs\eta,t)$ is separable of degree $d$ in $t$ and $\Gal(F(\bs\eta,t)/\F_q(x))=S_d\times C_2$, so by Lemma \ref{lem: disjoint} we can write $F(\bs\eta,t)=c(t)g(x,t)$ with $g$ absolutely irreducible and separable, $\Gal(g/\F_q(x))=S_d$. Therefore if $f_0\in\mathcal F$ is uniformly random and $n$ is large enough, by Proposition \ref{prop: hilbert 2}  (note that since $p>2$ we have $|\mathcal A_i|\gg q^{2n/3}$ and the condition $|\mathcal A_i|\ge nq^{n/2}$ of Proposition \ref{prop: hilbert 2} is satisfied for large enough $n$) we have
\begin{multline}\label{eq: prob abs irred}\P(f_0'\neq 0,\,g=f_0'/\con_t(f_0')\mbox{ is absolutely irreducible and separable in }t,\,\Gal(g/\F_q(x))=S_d)\\ \ge \frac{|\mathcal H_F\cap(\mathcal A_0\times\cdots\times\mathcal A_d)|}{|A_0\times\cdots\times\mathcal A_d|}=1+O\left(\max_{0\le i\le d}\frac{nq^{n/2}}{|\mathcal A_i|}\right)=1+O(nq^{-n/6}),\end{multline}
where in the last equality we used once again that $p>2$ and therefore $|\mathcal A_i|\gg q^{2n/3}$.
By Lemma \ref{lem: star} we also have \begin{equation}\label{eq: prob der is morse}\P\left(f_0\mbox{ satisfies }(\bigstar)\right)=1+O(q^{-n/2}),\end{equation} and by Lemma \ref{lem: lattice} we have
\begin{equation}\label{eq: prob deg c}\P\left(\deg\con_t(f')\ge n^{1/2}\right)=O\left(q^{-\lfloor n^{1/2}\rfloor}\right)=O\left(q^{-n^{1/3}}\right).\end{equation}
Combining (\ref{eq: prob abs irred}), (\ref{eq: prob deg c}), (\ref{eq: prob der is morse}) we obtain the assertion of Proposition \ref{prop: derivative}.

\end{proof}

\section{Discriminants in $\mathcal F_{f_0}$: preliminaries}
\label{sec: disc}

In the remainder of the paper we prove Proposition \ref{prop: disc}. In the present section we develop the necessary preliminaries before carrying out the main sieve argument in \S\ref{sec: sieve}. The proof will be concluded in \S\ref{sec: conclusion}.

Throughout this section $q$ is a fixed power of a prime $p>2$, $d$ a fixed natural number and $n$ is a parameter eventually taken to infinity. All asymptotic notation has implicit constants or rate of convergence which may depend on $q,d$.

We are assuming the reader's familiarity with the basic theory of univariate function fields. Good sources on the subject are \cite{Ros02} and \cite{Sti09}. Henceforth by a function field we will always mean a univariate function field with field of constants $\F_q$. For a function field $M$ we denote by $\mathcal P_M$ its set of prime divisors. If $M/N$ is a finite extension of function fields and $\mathfrak P\in\mathcal P_M$ lies over $\mathfrak p\in\mathcal P_N$, we will denote this by $\mathfrak P|\mathfrak p$. We will also denote the relative degree of $\mathfrak P$ in $M/N$ by $\deg_N\mathfrak P=\frac{\deg\mathfrak P}{\deg\mathfrak p}$. For $\mathfrak p\in\mathcal P_M, f\in M$ we denote by $v_{\mathfrak p}(f)$ the corresponding valuation. The height $\mathrm{ht}_M(f)$ is the degree of the polar divisor of $f$.

We now recall the setup of Proposition \ref{prop: disc}. Let \begin{equation}\label{eq: f0}f_0=x^n+\sum_{i=0}^{n-1}a_i(t)x^i,\quad a_i\in\F_q[t]_{\le d}\end{equation} be such that $f_0'=c(x)g(t,x)\neq 0$ with $g$ absolutely irreducible and separable in $t$ and $\Gal(g/\F_q(x))=S_d$. We are also assuming that $f_0$ satisfies Condition ($\bigstar$) of \S\ref{sec: outline thm2}: 
\begin{equation}\label{eq: star}f=\sum_{i=0}^db_i(x)t^i,\quad d=1\quad\mbox{or}\quad b_{1}\neq 0\mbox{ and }\frac{b_1}{(b_1,b_{2})}\not\in\F_q[x^p].\end{equation}
We will see below (Lemma \ref{lem: gamma}(i)) that this assumption guarantees that $\F_q(t,\alpha)=\F_q(t,f_0(t,\alpha))$ and that $f_0(t,\alpha)$ is not a $p$-th power in $\F_q(t,\alpha)$, where $\alpha$ is a root of $g\in\F_q(t)[x]$ in $\overline{\F_q(t)}$. This will be important in the sequel.

Our goal is to show that for a uniformly random $h\in\F_q[x]_{\le (n-1)/p}$, if we make the additional assumption $\deg c<n^{1/2}$ then \begin{equation}\label{eq: key estimate}\P\left(\exists P\in\F_q[t]\mbox{ prime}:\,n^{2/3}\le\deg P\le \frac n{3p\log^2n},\,P\Vert\disc(f_0(t,x)+h(x)^p)\right)\\=1+O\left(\frac{\log^{2}n}{n^{1/3}}\right).\end{equation}
From now on $f_0$ will be of the form (\ref{eq: f0}) with $f_0'=c(x)g(t,x)$, $g$ absolutely irreducible and separable in $t$ with $\Gal(g/\F_q(x))=S_d$ and we also assume (\ref{eq: star}). We do not yet require $\deg c<n^{1/2}$, this condition will come into play in \S\ref{sec: conclusion}.
In what follows $h$ will always be a random polynomial drawn uniformly from  $\F_q[x]_{\le (n-1)/p}$.

Let $\alpha\in\overline{\F_q(t)}$ be a root of $g$, $$L=\F_q(t,\alpha),\quad F=\F_q(t),\quad K=\F_q(\alpha).$$ Note that since $g$ is absolutely irreducible, the field of constants of $L$ is $\F_q$. Let $\mathcal O$ be the integral closure of $\F_q[t,\alpha]$ in $L$. We will identify prime divisors $P$ of $F=\F_q(t)$ satisfying $v_P(t)\ge 0$ with (monic) prime polynomials in $\F_q[t]$ in the usual way, and do the same for $K=\F_q(\alpha)$. We also identify the prime divisors $\mathfrak p\in\mathcal P_L$ satisfying $v_\mathfrak p(\alpha),v_\mathfrak p(t)\ge 0$ with the nonzero prime ideals of $\mathcal O$ via $\mathfrak p\mapsto\{f\in\mathcal O:\,v_\mathfrak p(f)\ge 1\}$.

Let $H\in K[T]=\F_q(\alpha)[T]$ be the monic minimal polynomial of $f_0(t,\alpha)$ over $K$ and denote $$D=\disc(H)=\res_T\left(H,\frac{\partial H}{\partial T}\right)\in K,\quad R=\res_T\left(H,\frac{\partial H}{\partial \alpha}\right)\in K$$
(resultant in the variable $T$). Since $g(T,\alpha)\in K[T]$ is a minimal polynomial of $t$ over $K$ and is irreducible and separable in $t$, the extension $L/K$ is separable and therefore $H$ is separable and $D\neq 0$.

\begin{lem}\label{lem: gamma}
\begin{enumerate}\item[(i)] $L=K(f_0(t,\alpha))$.
\item[(ii)] $f_0(t,\alpha)\not\in L^p=\{a^p:\,a\in L\}.$
\item[(iii)] $R\neq 0$.
% \item[(iv)] Let $Q\in\F_q[\alpha]$ be a prime such that denoting $A=\F_q[\alpha]_Q$ (localization at $Q$) we have $H,\frac{\partial H}{\partial\alpha}\in A[T]$ and $D,R\in A^\times$. Then $Q$ is unramified in $L/K$.
\item[(iv)] Let $Q\in\F_q[\alpha]$ be a prime and denote $A=\F_q[\alpha]_Q$ (localization at $Q$). Assume that $H,\frac{\partial H}{\partial\alpha}\in A[T]$ and $v_Q(D),v_Q(R)\in A^{\times}$. Let
$\mathfrak q\in\mathcal P_L,\,\mathfrak q| Q$
and assume that 
$v_{\mathfrak q}(f_0(t,\alpha)+u^p)\ge 1$ 
for some $u\in\F_q[\alpha]$. Then $\deg_K\mathfrak q=1$ and $v_{\mathfrak q}(f_0(t,\alpha)+u^p)=1$.\end{enumerate}\end{lem}

\begin{proof} Denote $\gamma=f_0(t,\alpha)$.

{\bf (i).} As noted above the extension $L/K$ is separable. We have $L=K(t)$, a (not necessarily monic) minimal polynomial of $t$ over $K$ is $U(T)=g(T,\alpha)\in K[T]$ and hence $\Gal(U/K)\cong\Gal(g/\F_q(x))=S_d$ (the latter equality was another assumption we made on $g$). The group $\Gal(U/K)=S_d$ acts primitively on the $d$ roots of $U$ and therefore there are no intermediate fields between $K$ and $L$. Thus either $L=K(\gamma)$ and we are done, or $\gamma\in K$. If $d=1$ we have $L=K$ and the assertion is trivial. It remains to show that if $d>1$ then $\gamma\not\in K$.

Assume by way of contradiction that $d>1$ and $\gamma\in K$. We will show that this contradicts the assumption (\ref{eq: star}). Write $\gamma=w(\alpha)$ for some $w\in\F_q(x)$. Since $g(T,\alpha)$ and equivalently $f_0'(T,\alpha)=c(\alpha)g(T,\alpha)\in K[T]$, is a minimal polynomial of $t$ over $K$, we have
$f_0(T,\alpha)+v(\alpha)f_0'(T,\alpha)=w(\alpha)$ for some $v\in\F_q(x)$ (here we used $\deg_tf_0=\deg_tf_0'=\deg_tg=d$, so the coefficient $v(\alpha)$ does not depend on $T$), equivalently
$$f_0(t,x)+v(x)f_0'(t,x)=w(x),\quad v,w\in\F_q(x).$$
Plugging in $f_0=\sum_{i=0}^db_it^i,\,f_0'=\sum_{i=0}^db_i't^i$  and recalling $d>1$, we obtain $b_1=-vb_1',\,b_2=-vb_2'$. Since by (\ref{eq: star}) we have $b_1,b_2\neq 0$, it follows that
$$\left(\frac{b_1}{b_2}\right)'=\frac{b_1'b_2-b_1b_2'}{b_2^2}=0,$$
which implies $\frac{b_1}{b_2}\in\F_q(x^p)$. But this contradicts (\ref{eq: star}).

{\bf (ii)} Assume by way of contradiction that $f_0(t,\alpha)\in L^p$. Then we may write $f_0(t,\alpha)=\frac{w(t^p,\alpha^p)}{z(\alpha^p)}$ for some $w\in\F_q[t,x],0\neq z\in\F_q(x)$ (because $K[t,\alpha]=L$). Arguing as in part (i) we may write
$$z(x^p)f_0(t,x)+z(x^p)v(x)f_0'(t,x)=w(t^p,x^p),\quad u,v\in\F_q(x).$$ Comparing coefficients at $t,t^2$ and using $p>2$ we once again deduce $\left(\frac{b_1}{b_2}\right)'=0$, contradicting (\ref{eq: star}).

{\bf (iii)} Since $H\in K[T]$ is monic and irreducible, we have $\deg_T\frac{\partial H}{\partial\alpha}<\deg_T H$ ($\frac{\partial}{\partial\alpha}$ kills the leading term) and therefore $R=\res_T\left(H,\frac{\partial H}{\partial\alpha}\right)\neq 0$ iff  $\frac{\partial H}{\partial\alpha}\neq 0$. Assume by way of contradiction that $\frac{\partial H}{\partial\alpha}=0$, so each coefficient of $H$ is a $p$-th power. By (i) we have $\deg H=d$ and we may write $H=\sum_{i=1}^da_i^pT^i,\,a_i\in K$. Then $\gamma^{1/p}$ is a root of $\sum_{i=0}^da_iT^i$ and therefore $[K(\gamma^{1/p}):K]\le d=[K(\gamma):K]$ and $\gamma^{1/p}\in K(\gamma)=L$, contradicting (ii).

% {\bf (iv).} Since $H\in A[T]$ is the (monic) minimal polynomial of $\gamma $ over $K$, the element $\gamma$ is integral over $A$. Since $L=K(\gamma)$ it follows (e.g. from \cite[Corollaries III.6.2, III.5.1]{Ser79}) that $A[\gamma]$ is the integral closure of $A$ in $L$ and $Q$ is unramified in $L/K$.

{\bf (iv).} Since $H\in A[T]$ is the (monic) minimal polynomial of $\gamma $ over $K$, the element $\gamma$ is integral over $A$. Since by assumption $D=\disc(H)\in A^\times$ and $L=K(\gamma)$ it follows (e.g. by \cite[Corollary III.6.2]{Ser79}) that $A[\gamma]$ is the integral closure of $A$ in $L$.
Consequently $A[\gamma]=(A\smallsetminus QA)^{-1}\mathcal O$ and the inclusion $\mathcal O\subset A[\gamma]$ induces an isomorphism $\mathcal O/\mathfrak q\isom A[\gamma]/\mathfrak qA[\gamma]$ (since $\mathfrak q\cap A=QA$ and $\mathcal O/\mathfrak q$ is a field).

The assumption $v_{\mathfrak q}(\gamma+u^p)\ge 1,\,u\in\F_q[\alpha]$ implies $\gamma\equiv -u^p\pmod{\mathfrak q}$ for some $u\in\F_q[\alpha]\subset A$ and hence $A[\gamma]=A+\mathfrak qA[\gamma]$ and $\mathcal O/\mathfrak q\cong A[\gamma]/\mathfrak qA[\gamma]\cong A/Q\cong\F_q[\alpha]/Q$ (using $QA=A\cap\mathfrak qA[\gamma]$ for the second isomorphism). Therefore $\deg\mathfrak q=\log_q|\mathcal O/\mathfrak q|=\log_q|\F_q[\alpha]/Q|=\deg Q$ and $\deg_K\mathfrak q=1$. Equivalently, $N_{L/K}(\mathfrak q)=Q$.

It remains to show $v_{\mathfrak q}(\gamma+u^p)=1$. Assume by way of contradiction that $v_{\mathfrak q}(\gamma+u^p)>1$, equivalently $\mathfrak q^2\,|\,\gamma+u^p$. Taking norms we obtain $Q^2\,|\,H(-u^p)$. Write $H(-u^p)=Q^2W,\,W\in\F_q[\alpha]$ and take the derivative with respect to $\alpha$:
$$Q\left(2\frac{\partial Q}{\partial\alpha}W+Q\frac{\partial W}{\partial\alpha}\right)=\frac{\partial H}{\partial\alpha}(-u^p).$$
We see that $-u^p$ is a common root of $H,\frac{\partial H}{\partial\alpha}$ modulo $Q$ in $A$, hence $R=\res_T(H,\frac{\partial H}{\partial\alpha})\in QA$, contradicting the assumption $R\in A^\times$.

\end{proof}

Next we define a set of ``good" primes of $L$ that will be used to carry out the sieve in \S\ref{sec: sieve}.

\begin{defi} Let $\mathcal S\subset\mathcal P_L$ be the set of prime divisors $\mathfrak p\in\mathcal P_L$ satisfying the following two conditions:
\begin{enumerate}
    \item[(a)] $\deg_F\mathfrak p=\deg_K\mathfrak p=1$. 
    \item[(b)]
    Any $\mathfrak q\in\mathcal P_L$ lying over the same prime divisor $P\in\mathcal P_F$ as $\mathfrak p$ satisfies $v_{\mathfrak q}(\alpha),v_{\mathfrak q}(t)\ge 0$ and if $Q\in\F_q[\alpha]$ is the prime of $K$ lying under $\mathfrak q$ then $H,\frac{\partial H}{\partial\alpha}\in\left(\F_q[\alpha]_Q\right)[T],\,v_Q(D)=v_Q(R)=1$ (i.e. $Q$ satisfies the assumptions of Lemma \ref{lem: gamma}(iv)).
\end{enumerate}
We also denote $$\mathcal S_m=\{\mathfrak p\in\mathcal S:\,\deg\mathfrak p=m\}.$$
\end{defi}

By condition (a) in the definition of $\mathcal S$, for any $\mathfrak p\in S$ we have $P=N_{L/F}(\mathfrak p)\in\mathcal P_F,\,Q=N_{L/K}(\mathfrak p)\in\mathcal P_K,\,\deg P=\deg Q=\deg\mathfrak p$ and (since condition (b) implies $v_{\mathfrak p}(t),v_{\mathfrak p}(\alpha)\ge 0$) $P,Q$ can be identified with monic prime polynomials in $\F_q[t],\F_q[\alpha]$ respectively. In particular condition (b) is well-defined. Since $v_{\mathfrak p}(t),v_{\mathfrak p}(\alpha)\ge 0$ the prime divisor $\mathfrak p$ can be identified with a prime ideal of $\mathcal O$. In a suitable sense most primes in $\mathcal P_L$ are in $\mathcal S$ (see Lemma \ref{lem: S_m count} below) and we aim to show the following:
\begin{enumerate}
\item[(i)] With probability $1+O\left(\frac{\log^{2}n}{n^{1/3}}\right) $ there exists $n^{2/3}\le m\le\frac n{3p\log^2n}$ such that $f_0(t,\alpha)+h(\alpha)^{p}$ is divisible by a unique $\mathfrak p\in\mathcal S_m$.
\item[(ii)] If the above event occurs then with probability $1+O\left(\frac{\log^{2}n}{n^{1/3}}\right) $ we have $P\,\Vert\,\disc(f_0(t,x)+h(x)^p)$, where $P=N_{L/F}(\mathfrak p)$.
\end{enumerate}
Some useful properties of the primes in $\mathcal S$ are summarized in the following

\begin{lem}\label{lem: crt} \begin{enumerate}
\item[(i)]Let $\mathfrak p_1,\ldots,\mathfrak p_k\in\mathcal S$ be distinct with $e=\sum_{i=1}^k\deg\mathfrak p_i\le \frac{n-1}p$. Then
$$\P\left(\mathfrak p_i\mbox{ divides }\,f_0(t,\alpha)+h(\alpha)^{p}\mbox{ for each }1\le i\le k\right)=q^{-e}.$$
\item[(ii)] Let $\mathfrak p\in\mathcal S$.  Then $$\P\left(\mathfrak p^{2}\mbox{ divides } f_0(t,\alpha)+h(\alpha)^{p}\right)=0.$$
\item[(iii)] Let $\mathfrak p\in\mathcal S,\,P=N_{L/F}(\mathfrak p)$ (recall $P\in\mathcal P_F$) and $\mathfrak q\in\mathcal P_L$ a prime divisor lying over $P$. Then
$$\P\left(\mathfrak q\mbox{ divides }\,f_0(t,\alpha)+h(\alpha)^{p}\right)\le q^{-{\min\left(\deg\mathfrak q,\frac{n-1}{p}\right)}}.$$
\end{enumerate}\end{lem}

\begin{proof}
{\bf (i).} By the Chinese Remainder Theorem the assertion is equivalent to
\begin{equation}\label{eq: primes divide indep}\P\left(\prod_{i=1}^k\mathfrak p_i\mbox{ divides }\,f_0(t,\alpha)+h(\alpha)^p\right)=q^{-e}.\end{equation}
Since $\mathfrak p_i\in\mathcal S$, we have $\deg_K\mathfrak p_i=1$ and $Q_i=N_{L/K}(\mathfrak p_i)\in\F_q[\alpha]$ is prime with $\deg Q_i=\deg \mathfrak p_i$. There is an isomorphism $\F_q[\alpha]/Q_i\isom\mathcal O/\mathfrak p_i$ induced by the inclusion $\F_q[\alpha]\subset\mathcal O$. Composing with the Frobenius if necessary and using the Chinese Remainder Theorem we see that $h(\alpha)\bmod \prod_{i=1}^kQ_i\mapsto h(\alpha)^{p}\bmod\prod_{i=1}^k\mathfrak p_i$ is an isomorphism $\F_q[\alpha]/\prod_{i=1}^kQ_i\isom\mathcal O/\prod_{i=1}^k\mathfrak p_i$.
Since we assumed $e=\sum_{i=1}^k\deg\mathfrak p_i=\sum_{i=1}^k\deg Q_i\le\frac{n-1}p$ and $h$ is picked uniformly from $\F_q[x]_{\le (n-1)/p}$, the residue $h\bmod\prod_{i=1}^kQ_i$ is uniform and therefore so is $h(\alpha)^{p}\bmod\prod_{i=1}^k\mathfrak p_i$. The claim (\ref{eq: primes divide indep}) follows.

{\bf (ii).} 
By condition (b) in the definition of $\mathcal S$ the assumptions of Lemma \ref{lem: gamma}(iv) are satisfied for $Q=N_{L/K}(\mathfrak p)$. Therefore by Lemma \ref{lem: gamma}(iv) we cannot have $v_{\mathfrak p}(f_0(t,\alpha)+h(\alpha)^p)>1$, equivalently we cannot have $\mathfrak p^2\,|\,f_0(t,\alpha)+h(\alpha)^p$.

{\bf (iii).} Let $\mathfrak p\in\mathcal S,\,P=N_{L/F}(\mathfrak p)$, $\mathfrak q\in\mathcal P_L$ a prime divisor lying over $P$. By condition (b) in the definition of $\mathcal S$, the assumptions of Lemma \ref{lem: gamma}(iv) are satisfied and therefore by Lemma \ref{lem: gamma}(iv) we have $\deg_K\mathfrak q=1$ if $\mathfrak q\,|\,f_0(t,\alpha)+h(\alpha)^{p}$ for any $h\in\F_q[x]$. Hence we may assume $\deg_K\mathfrak q=1$ (otherwise the asserted probability is 0).

Now if $\deg\mathfrak q\le\frac{n-1}p$ we may argue as in part (i) to obtain $\P(\mathfrak q\,|\,f_0(t,\alpha)+h(\alpha)^{p})=q^{-\deg\mathfrak q}$. If on the other hand $\deg\mathfrak q>\frac{n-1}p$ then denoting $Q=N_{L/K}(\mathfrak q)\in\mathcal P_K$ we have (as in part (i)) an isomorphism $\F_q[\alpha]/Q\isom\mathcal O/\mathfrak q$ given by $h(\alpha)\bmod Q\mapsto h(\alpha)^{p}\bmod\mathfrak q$. Since $\deg Q=\deg\mathfrak q>\frac{n-1}p$, the map $h\mapsto \gamma+h(\alpha)^{p}\bmod\mathfrak q$ is injective on $\F_q[x]_{\le(n-1)/p}$ and we obtain
$$\P\left(\mathfrak q\,|\,f_0(t,\alpha)+h(\alpha)^{p}\right)\le \left|\F_q[x]_{\le(n-1)/p}\right|^{-1}=q^{-\lfloor\frac{n-1}p\rfloor-1}\le q^{-\frac{n-1}p},$$
as required.

\end{proof}

\begin{lem}\label{lem: S_m count}$|\mathcal S_m|=\frac{q^m}m+O({n^2}q^{m/2})$ whenever $m\ge 2\log_qn$. \end{lem}

\begin{proof} First denote $\mathcal P_{L,m}=\{\mathfrak p\in\mathcal P_L:\,\deg\mathfrak p=m\}$. Since we assumed that $g(t,x)$ is absolutely irreducible, there exists a smooth projective curve $C_L$ defined over $\F_q$ with function field $L$. The curve $C_L$ has an affine plane model given by the equation $g(t,x)=0$. Since $\deg g=O(n)$, the genus of $C_L$ is $O(n^2)$. Therefore by the Weil bound
$$|\mathcal P_{L,m}|=\frac{q^m}m+O\left(\frac{n^2}mq^{m/2}\right)=\frac {q^m}m+O(n^2q^{m/2}).$$
It remains to check that for each of the conditions appearing in the definition of $\mathcal S$, the number of $\mathfrak p\in\mathcal P_{L,m}$ violating the condition is $O(n^2q^{m/2})$.

First we bound the number of $\mathfrak p\in\mathcal P_{L,m}$ violating condition (a), i.e. satisfying $\deg_F\mathfrak p>1$ or $\deg_K\mathfrak p>1$. 
If $\mathfrak p\in\mathcal P_{L,m}$ and say $\deg_F\mathfrak p>1$, denote by $P$ the prime of $\mathcal P_F$ lying under $\mathfrak p$. Then 
$\deg P=\deg\mathfrak 
p/\deg_F\mathfrak p\le m/2$ and there are $O(q^{m/2})$ possible values of $P$, over 
each of which lie at most $n$ primes of $L$ (since $[L:F]\le n$). Overall we 
have $O(nq^{m/2})$ values of $\mathfrak p\in\mathcal P_{L,m}$ such 
that $\deg_F\mathfrak p>1$ and 
similarly (using $[L:K]\le d=O(1)$) there are at most 
$O(q^{m/2})$ primes 
$\mathfrak p\in\mathcal P_{L,m}$ with $\deg_K\mathfrak p>1$. In total the number of $\mathfrak p\in\mathcal P_{L,m}$ violating condition (b) is $O(nq^{m/2})$, which is negligible.

Next we bound the number of $\mathfrak p\in\mathcal S_m$ violating condition (b).
Since $[L:F]\le\deg_xg=O(n),\,[L:K]=\deg_t g=d=O(1)$, we have
$$\mathrm{ht}_L(t)=[L:F]\cdot\mathrm{ht}_F(t)=[L:F]=O(n),$$
$$\mathrm{ht}_L(\alpha)=[L:K]\cdot\mathrm{ht}_K(\alpha)=[L:K]=O(1),$$
$$\mathrm{ht}_L(f_0(t,\alpha))\le (\deg f_0)\cdot(\mathrm{ht}_L(t)+\mathrm{ht}_L(\alpha))=O(n^2)$$
and since $H$ is the monic minimal polynomial of $f_0(t,\alpha)$ over $K$ (and $[L:K]=O(1)$) its coefficients have height $O(n^2)$. Consequently (using $\deg H=O(1)$) we have $\mathrm{ht}_LD,\mathrm{ht}_L(R)=O(n^2)$.
Hence the number of primes of $L$ at which one of $\alpha,t,D,R$ or a coefficient of $H$ or $\frac{\partial H}{\partial\alpha}$ has a zero or a pole is $O(n^2)$. If $\mathfrak q$ is such a prime, the number of $\mathfrak p\in\mathcal P_L$ lying over the same $P\in\mathcal P_F$ as $\mathfrak q$ is $\le[L:F]=O(n)$. In total we have $O(n^3)$ primes $\mathfrak p$ violating condition (b), which is $O(n^2q^{m/2})$ provided $m\ge 2\log_qn$.

\end{proof}
\noindent

\section{Discriminants in $\mathcal F_{f_0}$: the sieve}
\label{sec: sieve}

In the present section we maintain the setup and notation of \S\ref{sec: disc}.
\\ \\
{\bf Notation.} For any natural number $m$ denote by $E_m$ the event that $f_0(t,\alpha)+h(\alpha)^{p}$ has a unique prime divisor $\mathfrak p\in \mathcal S_m$. Denote by $E_m'$ the event that $\disc_x(f_0(t,x)+h(x)^p)\in\F_q[t]$ has a prime divisor of degree $m$ with multiplicity $1$.
\\

In the present section we establish by an elementary sieve argument that with probability close to 1, one of the events $E_m$ with $n^{2/3}\le m\le\frac n{3p\log^2n}$ occurs. In the next section we will see how this implies the same for $E_m'$ by showing that $\sum_{n^{2/3}\le m\le\frac{n}{3p\log^2n}}\P(E_m\setminus E_m')$ is small. This will give us (\ref{eq: key estimate}), completing the proof of Proposition \ref{prop: disc}.

\begin{lem}\label{lem: cap E} Let $k,M,N$ be natural numbers such that $8\log_qn\le M\le N\le \frac {n-1}{2pk}$ and $k\le M$. For any $M\le m_1<m_2<\ldots<m_k\le N$ we have
$$\P(E_{m_1}\cap\ldots\cap E_{m_k})=\frac 1{\prod_{i=1}^km_i}\left(1+O\left(\frac kM\right)\right).$$

\end{lem}

\begin{proof}
    For each $\mathfrak p\in\mathcal S$ denote by $A_{\mathfrak p}$ the event that $\mathfrak p\,|\,f_0(t,\alpha)+h(\alpha)^p$. Using the shorthand 
    $$\underline{\mathfrak p}=(\mathfrak p_1,\ldots,\mathfrak p_k),\quad A_{\underline{\mathfrak p}}=A_{\mathfrak p_1}\cap\ldots \cap A_{\mathfrak p_k},\quad\mathcal S_{\underline m}=\mathcal S_{m_1}\times\cdots\times S_{m_k}$$ we have
    \begin{equation}\label{eq: cap E as union}E_{m_1}\cap\ldots\cap E_{m_k}=\Cup_{\underline{\mathfrak p}\in\mathcal S_{\underline m}}A_{\underline{\mathfrak p}}
    \setminus
    \Cup_{\underline{\mathfrak p},\underline{\mathfrak q}\in\mathcal S_{\underline m}\atop{\underline{\mathfrak p}\neq\underline{\mathfrak q}}}(A_{\underline{\mathfrak p}}\cap A_{\underline{\mathfrak q}}).\end{equation}
and by inclusion-exclusion
\begin{equation*}\sum_{\underline{\mathfrak p}\in\mathcal S_{\underline m}}\P(A_{\underline{\mathfrak p}})-\frac 12\sum_{\underline{\mathfrak p},\underline{\mathfrak q}\in\mathcal S_{\underline m}\atop{\underline{\mathfrak p}\neq\underline{\mathfrak q}}}\P(A_{\underline{\mathfrak p}}\cap A_{\underline{\mathfrak q}})\le \P\left(\Cup_{\ul{\mf p}\in\mc S_{\ul m}}A_{\ul{\mf p}}\right)\le \sum_{\underline{\mathfrak p}\in\mathcal S_{\underline m}}\P(A_{\underline{\mathfrak p}}).\end{equation*}
Combined with (\ref{eq: cap E as union}) this gives
\begin{equation}\label{eq: inc exc}\sum_{\underline{\mathfrak p}\in\mathcal S_{\underline m}}\P(A_{\underline{\mathfrak p}})-\sum_{\underline{\mathfrak p},\underline{\mathfrak q}\in\mathcal S_{\underline m}\atop{\underline{\mathfrak p}\neq\underline{\mathfrak q}}}\P(A_{\underline{\mathfrak p}}\cap A_{\underline{\mathfrak q}})\le \P(E_{m_1}\cap\ldots \cap E_{m_k})\le \sum_{\underline{\mathfrak p}\in\mathcal S_{\underline m}}\P(A_{\underline{\mathfrak p}}).\end{equation}

By Lemma \ref{lem: crt}(i) and the assumption $N\le \frac{n-1}{2pk}$, which implies $m_1+\ldots+m_k\le Nk\le \frac{n-1}p$, we have 
\begin{equation}\label{eq: prob A_p}\P(A_{\underline{\mathfrak p}})=\P(A_{\mathfrak p_1}\cap\ldots\cap A_{\mathfrak p_k})=q^{-\sum_{i=1}^km_i}\end{equation}
whenever $\underline{\mathfrak p}\in\mathcal S_{\underline m}$ (i.e. $\mathfrak p_i\in\mathcal S_{m_i},\,1\le i\le k$). Similarly, if for $\ul{\mf p},\ul{\mf q}\in\mc S_{\ul m}$ we denote 
$$I(\ul{\mf p},\ul{\mf q})=\{1\le i\le k:\,\mf p_i\neq\mf q_i\}$$
then
\begin{equation}\label{eq: prob A_pq}\P(A_{\ul{\mf p}}\cap A_{\ul{\mf q}})=\P\left(\bigcap_{i=1}^kA_{\mf p_i}\cap\bigcap_{j\in I(\ul{\mf p},\ul{\mf q})}A_{\mf q_j}\right)=q^{-\sum_{i=1}^km_i-\sum_{j\in I(\ul{\mf p},\ul{\mf q})}m_j}\end{equation} (Lemma \ref{lem: crt}(i) applies because $2\sum_{i=1}^km_i\le 2Nk\le\frac{n-1}p$).

Using (\ref{eq: prob A_p}) and Lemma \ref{lem: S_m count} we obtain
\begin{multline}\label{eq: sum size cap A_p}\sum_{\underline{\mathfrak p}\in\mathcal S_{\underline m}}\P(A_{\underline{\mathfrak p}})=q^{-\sum_{i=1}^km_i}|\mathcal S_{\ul m}|=q^{-\sum_{i=1}^km_i}\cdot\prod_{i=1}^k|\mathcal S_{m_i}|=q^{-\sum_{i=1}^km_i}\frac {q^{\sum_{i=1}^km_i}}{\prod_{i=1}^km_i}\prod_{i=1}^k(1+O(n^2q^{-{m_i/2}}))
\\= \frac{1}{\prod_{i=1}^km_i}(1+O(n^2q^{-M/2}))^n=\frac{1}{\prod_{i=1}^km_i}\left(1+O\left(\frac 1n\right)\right)\end{multline}
(in the last equality we used the condition $M\ge 8\log_q n$). Similarly, using (\ref{eq: prob A_pq}) we obtain
\begin{multline}\label{eq: sum size cap A_pq}
 \sum_{\ul{\mf p},\ul{\mf q}\in\mc S_{\ul m}\atop{\ul{\mf p}\neq\ul{\mf q}}}\P(A_{\ul{\mf p}}\cap A_{\ul{\mf q}})=\sum_{\emptyset\neq I\subset\{1,\ldots,k\}}\sum_{\ul{\mf p},\ul{\mf q}\in\mc S_{\ul m}\atop{I(\ul{\mf p},\ul{\mf q})=I}}\P(A_{\ul{\mf p}}\cap A_{\ul{\mf q}})=
 \sum_{\emptyset\neq I\subset\{1,\ldots,k\}}\sum_{\ul{\mf p},\ul{\mf q}\in\mc S_{\ul m}\atop{I(\ul{\mf p},\ul{\mf q})=I}}q^{-\sum_{i=1}^km_i-\sum_{j\in I}m_j}
 \\=\sum_{\emptyset\neq I\subset\{1,\ldots,k\}}q^{-\sum_{i=1}^km_i-\sum_{j\in I}m_j}\prod_{i=1}^k|\mathcal S_{m_i}|\prod_{j\in I}|\mathcal S_{m_j}|=\sum_{\emptyset\neq I\subset\{1,\ldots,k\}}\frac 1{\prod_{i=1}^km_i\prod_{j\in I}m_j}\left(1+O\left(\frac 1n\right)\right)
 \\ \ll
 \frac 1{\prod_{i=1}^km_i}\sum_{\emptyset\neq I\subset\{1,\ldots,k\}}\frac 1{M^{|I|}}\ll\frac 1{\prod_{i=1}^km_i}\sum_{l=1}^k\left(\begin{array}{cc}k\\l\end{array}\right) \frac1{M^l}\le \frac 1{\prod_{i=1}^km_i}\left(1+\frac 1M\right)^k=\frac 1{\prod_{i=1}^km_i}\cdot O\left(\frac kM\right).
\end{multline}

Combining (\ref{eq: inc exc}),(\ref{eq: sum size cap A_p}),(\ref{eq: sum size cap A_pq}) we obtain the assertion of the lemma.

\end{proof}

\begin{prop}\label{prop: cup E} Let $\log^2n\le M\le N\le \frac{n-1}{2p\log^2 n}$ be integers. Then $$\P\left(\Cup_{M\le m\le N}E_m\right)=1+O\left(\frac MN+\frac{N\log^2N}{M^2}\right).$$ \end{prop}

\begin{proof} Let $l\le\log^2n$ be an even natural number. By inclusion-exclusion and using Lemma \ref{lem: cap E} (note that for large enough $n$ and any $k\le l$ the conditions of Lemma \ref{lem: cap E} hold) we have
\begin{multline}\label{eq: prob cup E}
\P\left(\Cup_{M\le m\le N}E_m\right)\ge 
\sum_{k=1}^l(-1)^{k+1}\sum_{M\le m_1<m_2<\ldots<m_k\le N}\P(\cap_{i=1}^kE_{m_i})\\
=\sum_{k=1}^l(-1)^{k+1}\sum_{M\le m_1<m_2<\ldots<m_k\le N}\frac 1{\prod_{i=1}^km_i}\left(1+O\left(\frac lM\right)\right)\\
=\sum_{k=1}^l(-1)^{k+1}\sum_{M\le m_1<m_2<\ldots<m_k\le N}\frac 1{\prod_{i=1}^km_i}+O\left(\frac l{M}\sum_{k=1}^l\sum_{M\le m_1<m_2<\ldots<m_k\le N}\frac 1{\prod_{i=1}^km_i}\right).
\end{multline}
We have
$$\sum_{k=1}^l\sum_{M\le m_1<m_2<\ldots<m_k\le N}\frac 1{\prod_{i=1}^km_i}\le\prod_{M\le m\le N}\left(1+\frac 1m\right)\le e^{\sum_{m=M}^N\frac 1m}=e^{\log N-\log M+O(1)}=O\left(\frac NM\right)
$$
and hence the error term in (\ref{eq: prob cup E}) is $O\left(\frac{lN}{M^2}\right)$.
For any natural number $l$ denote
$$U(l)=\sum_{k=1}^l(-1)^{k+1}\sum_{M\le m_1<m_2<\ldots<m_k\le N}\frac 1{\prod_{i=1}^km_i}$$
Then for even $l$ we have 
$$U(l)\le 1-\prod_{M\le m\le N}\left(1-\frac 1m\right)\le U(l-1)$$
(e.g. apply inclusion-exclusion to a sequence of independent events $F_m,\,M\le m\le N$ with $\P(F_m)=\frac 1m$).
Now 
$$\prod_{M\le m\le N}\left(1-\frac 1m\right)\le e^{-\sum_{m=M}^N\frac 1m}= e^{\log M-\log N+O(1)}=O\left(\frac MN\right)$$
and
\begin{multline*}|U(l)-U(l-1)|\le\sum_{M\le m_1<m_2<\ldots<m_l\le N}\frac 1{\prod_{i=1}^lm_i}=\frac 1{l!}\sum_{M\le m_1,\ldots,m_l\le N\atop{\mathrm{distinct}}}\frac 1{\prod_{i=1}^lm_i}
\le\frac 1{l!}\left(\sum_{M\le m\le N}\frac 1m\right)^l\\
=O\left(\frac{(\log N-\log M+1)^l}{l!}\right).\end{multline*}
Therefore
$$U(l)=1+O\left(\frac MN+\frac{(\log N-\log M+1)^l}{l!}\right)$$
and we conclude that $$\P\left(\Cup_{M\le m\le N}E_m\right)=U(l)+O\left(\frac{lN}{M^2}\right)=1+O\left(\frac{(\log N-\log M+1)^l}{l!}+\frac MN+\frac{lN}{M^2}\right).$$
Taking $l=2\big\lfloor\frac{\log^2N}2\big\rfloor$ we obtain the desired result.
\end{proof}

\begin{cor}\label{cor: cup E} $$\P\left(\Cup_{n^{1/3}\le m\le\frac{n}{3p\log^2n}}E_m\right)=1+O\left(\frac{\log^2n}{n^{1/3}}\right).$$\end{cor}

\begin{proof} 
Take $M=n^{2/3},\,N=\big\lfloor\frac{n}{3p\log^2n}\big\rfloor$ in Proposition \ref{prop: cup E}.
\end{proof}

\section{Discriminants in $\mathcal F_{f_0}$: bounding $\sum\P(E_m\setminus E_m')$}
\label{sec: conclusion}

In the present section we conclude the proof of Proposition \ref{prop: disc} by showing that the discrepancy between $E_m$  and $E_m'$ is small when $m$ lies in a suitable range. We maintain the setup and notation of \S\ref{sec: disc}-\ref{sec: sieve}.

\begin{prop}\label{prop: Em to Em'} Denote $k=\deg c$. Then $\P(E_m\setminus E_m')=O(nq^{-m})$ whenever $\max(kd,8\log_qn)<m\le\frac n{2p}$.\end{prop}

\begin{proof} Assume that $E_m$ occurs. This means that $f_0(t,\alpha)+h(\alpha)^{p}$ has a unique prime divisor $\mathfrak p\in\mathcal S_m$. Denote $P=N_{L/F}(\mathfrak p)\in\mathcal P_F$. We have $\deg P=m$ (since $\deg_F\mathfrak p=1$). It is enough to show that \begin{equation}\label{eq: Em equiv}\P\left(E_m\mbox{ and not }P\,\Vert\,\disc(f_0(t,x)+h(x)^p)\right)=O(nq^{-m}),\end{equation}
since this event contains $E_m\setminus E_m'$.

Recall $f_0'(t,x)=c(x)g(t,x)$ and write $g(t,x)=\sum_{i=0}^s g_i(t)x^i,\,g_s\neq 0$. Let $\alpha_1,\ldots,\alpha_s\in\overline L$ be the roots of $g$ (counted with multiplicity), $\beta_1,\ldots,\beta_k\in\overline\F_q\subset L,\,k=\deg c$ the roots of $c(x)$. Using (\ref{eq: disc prod over roots}) and the irreducibility of $g$, we have
\begin{equation*}\begin{split}\disc\left(f_0(t,x)+h(x)^p\right)&=\const\cdot
g_s(t)^n\prod_{i=1}^s(f_0(t,\alpha_i)+h(\alpha_i)^p)\prod_{j=1}^k(f_0(t,\beta_j)+h(\beta_j)^p)\\
&=\const\cdot g_s(t)^nN_{L/F}\left(f_0(t,\alpha)+h(\alpha)^p\right)\prod_{j=1}^k\left(f_0(t,\beta_j)+h(\beta_j)^p\right).\end{split}\end{equation*}
Here $\const$ denotes an arbitrary constant in $\F_q^\times$.

First note that $\prod_{j=1}^k(f_0(t,\beta_j)+h(\beta_j)^p)\in\F_q[t]_{\le kd}$ because $\deg_t(f_0)\le d$ and $\beta_j\in\overline \F_q$, so by the assumption $m>kd$ this term is not divisible by $P$. The term $g_s(t)^n$ is also not divisible by $P$ since $\deg g_s\le \deg_tg=d<m$. Consequently we have $$v_P\left(\disc(f_0(t,x)+h(x)^p)\right)=v_P\left(N_{L/F}(f_0(t,\alpha)+h(\alpha)^{p})\right).$$
Therefore conditional on $E_m$, the negation of the event in (\ref{eq: Em equiv}) is equivalent to $v_P\left(N_{L/F}(f_0(t,\alpha)+h(\alpha)^{p})\right)=1$, which in turn (since $\mathfrak p\,|\,f_0(t,\alpha)+h(\alpha)^{p}$, $\deg_F\mathfrak p=1$ and any $\mathfrak q|P$ satisfies $v_{\mathfrak q}(t),v_{\mathfrak q}(\alpha)\ge 0$ by condition (b) in the definition of $\mathcal S$) is equivalent to the conjunction of the following two conditions:
\begin{enumerate}
\item[(i)] $v_{\mathfrak p}(f_0(t,\alpha)+h(\alpha)^{p})=1$,
\item[(ii)] $v_{\mathfrak q}(f_0(t,\alpha)+h(\alpha)^{p})=0$ for any $\mathfrak q\neq\mathfrak p$ lying over $P$.
\end{enumerate}

Since $\mathfrak p\,|\,f_0(t,\alpha)+h(\alpha)^{p}$, (i) holds by Lemma \ref{lem: crt}(ii).
If condition (ii) fails, there exists $\mathfrak q|P,\,\mathfrak q\neq\mathfrak p$ such that $\mathfrak q\,|\,f_0(t,\alpha)+h(\alpha)^{p}$. By condition (b) in the definition of $\mathcal S$ and Lemma \ref{lem: gamma}(iv) we have $\deg_K\mathfrak q=1$. If in addition $\deg_F\mathfrak q=1$, then $\mathfrak q\in\mathcal S$ and $\deg\mathfrak q=\deg P=m$, so $\mathfrak q\in\mathcal S_m$. This contradicts the uniqueness assumption in the definition of $E_m$. Therefore we must have $\deg_F\mathfrak q>1$. Hence using Lemma \ref{lem: crt}(iii), Lemma \ref{lem: S_m count}, $[L:F]=O(n)$ and the assumption $8\log_qn< m\le\frac n{2p}$,
\begin{multline*}\label{eq: cond ii fails}\P(E_m\mbox{ and (ii) fails})\le\sum_{\mathfrak p\in\mathcal S_m}\sum_{\mathfrak q|N_{L/F}(\mathfrak p)\atop{\deg_F\mathfrak q>1}}\P\left(\mathfrak q\,|\,f_0(t,\alpha)+h(\alpha)^{p}\right)\ll \frac {q^m}m\cdot n\cdot\max_{\mathfrak p\in\mathcal S_m\atop{\mathfrak q|N_{L/F}(\mathfrak p)\atop{\deg_F\mathfrak q>1}}}\left(q^{-\min(\deg\mathfrak q,n/p)}\right)\\
=\frac {nq^m}m\max_{\mathfrak p\in\mathcal S_m\atop{\mathfrak q|N_{L/F}(\mathfrak p)\atop{\deg_F\mathfrak q>1}}}\left(q^{-\min((\deg_F\mathfrak q) m,n/p)}\right) \le \frac {nq^m}mq^{-2m}\ll nq^{-m}.
\end{multline*}
Thus $$\P\left(E_m\mbox{ and not }P\,\Vert\,\disc(f_0(t,\alpha)+h(x)^p))\right)=\P(E_m\mbox{ and (ii) fails})=O(nq^{-m}),$$
establishing (\ref{eq: Em equiv}) and completing the proof.
\end{proof}

\begin{cor}\label{cor: Em to Em'} Assume $k=\deg c< n^{1/2}$. Then $\sum_{n^{2/3}\le m\le\frac{n}{3p\log^2n}}\P(E_m\setminus E_m')=O\left(nq^{-n^{2/3}}\right)$.\end{cor}

\begin{proof} Sum the assertion of Proposition \ref{prop: Em to Em'} over $n^{2/3}\le m\le\frac{n}{3p\log^2n}$.\end{proof}

\begin{proof}[Proof of Proposition \ref{prop: disc}]
By Corollary \ref{cor: cup E} combined with Corollary \ref{cor: Em to Em'},
\begin{multline*}\P\left(\Cup_{n^{2/3}\le m\le\frac n{3p\log^2n}}E_m'\right)\ge\P\left(\Cup_{n^{2/3}\le m\le\frac n{3p\log^2n}}E_m\right)-\sum_{n^{2/3}\le m\le\frac n{3p\log^2n}}\P(E_m\setminus E_m')\\=1+O\left(\frac{\log^2n}{n^{1/3}}\right)+O\left(nq^{-n^{2/3}}\right)=1+O\left(\frac{\log^2n}{n^{1/3}}\right).\end{multline*} This gives (\ref{eq: key estimate}), establishing Proposition \ref{prop: disc}.
\end{proof}

\bibliography{mybib}
\bibliographystyle{alpha}

\end{document}